\documentclass[11pt,reqno]{amsart}
\usepackage{amssymb,latexsym}
\usepackage{enumerate}

\newtheoremstyle{df}{8pt}{}{}{0cm}{\bf}{.\ }{0pt}{}
\numberwithin{equation}{section}
\newtheorem{tw}{Theorem}[section]
\newtheorem{lm}[tw]{Lemma}
\newtheorem{stw}[tw]{Proposition}
\newtheorem{wn}[tw]{Corollary}

\theoremstyle{df}

\renewenvironment{proof}[1][\proofname]{{\fontfamily{pcr}\selectfont #1.}}{\qed\newline}

\newcommand{\R}{\mathbb{R}}

\newcommand{\X}{\mathcal{X}}

\newcommand{\D}{\operatorname{D}}
\newcommand{\id}{\operatorname{id}}
\newcommand{\DD}{\mathcal{D}}

\newcommand{\Dh}{\mathcal{D}_h}

\begin{document}
\baselineskip=17pt
\title{$\widetilde{J}$-tangent affine hypersurfaces with an induced almost paracontact structure}
\author{Zuzanna Szancer}
%\address{Zuzanna Szancer\\
%            Instytut Matematyki\\
%         UJ\\
%        ul. Reymonta 4, 30-059 Kraków, Polska}
%\email{Zuzanna.Szancer@im.uj.edu.pl}
%
%\urladdr{http://...........}

%\thanks{}
\keywords{Affine hypersurface, almost paracontact structure, Sasakian structure, hyperquadric}
\subjclass[2010]{53A15, 53D15}
%\date{June 5, 2007}
\begin{abstract}
We study real affine hypersurfaces $f\colon M\rightarrow \mathbb{R}^{2n+2}$ with an almost paracontact structure $(\varphi ,\xi,\eta)$ induced by a $\widetilde{J}$-tangent transversal vector filed, where $\widetilde{J}$ is the canonical paracomplex structure on $\mathbb{R}^{2n+2}$. We give a classification of hypersurfaces for which an induced almost paracontact structure is metric relative to the second fundamental form. Some other properties of such hypersurfaces are also studied.
\end{abstract}

\maketitle

\section{Introduction}

\par Paracontact metric structures were introduced in \cite{11} by S. Kaneyuki and F. L. Williams.
The importance of paracontact geometry, and in particular of para-Sasakian geometry, has been indicated in the recent years by many authors.
Its role in pseudo-Riemannian geometry as well as in mathematical physics was emphasized in several papers (see e.g. \cite{1}, \cite{2}, \cite{6}, \cite{18} ).
Recently (\cite{IKE}) I. K\"upeli Erken studied normal almost paracontact metric manifolds provided they satisfy some additional projective flatness conditions.
\par Relations between affine differential geometry and paracomplex geometry can be found in \cite{7} and \cite{LS} for example.
Moreover, affine immersions with an almost product structures are also studied (see e.g. \cite{K}).

In \cite{SZM} the author studied affine hypersurfaces with an arbitrary $J$-tangent transversal vector field, where $J$ was the canonical complex structure on $\mathbb{R}^{2n+2}\cong\mathbb{C}^{n+1}$. It was proved that if the induced almost contact structure is metric relative to the second fundamental form then it is a Sasakian structure and the hypersurface itself is a piece of hyperquadric.
In this paper we study affine hypersurfaces $f\colon M\rightarrow \mathbb{R}^{2n+2}$ with an arbitrary
$\widetilde{J}$-tangent transversal vector field, where $\widetilde{J}$ is the canonical paracomplex structure on $\mathbb{R}^{2n+2}$. Such a vector field induces in a
natural way an almost paracontact structure $(\varphi,\xi,\eta)$ as well as
the second fundamental form $h$. We prove that if
$(\varphi,\xi,\eta,h)$ is an almost paracontact metric structure then it
is a para~$\alpha $-Sasakian structure with $\alpha =-1$. Moreover, the hypersurface is a piece of a hyperquadric.
\par In Section 2, we briefly recall the basic formulas of affine differential geometry. We introduce the notion of a $\widetilde{J}$-tangent transversal vector field and a $\widetilde{J}$-invariant distribution $\DD$.
\par In section 3 we recall the definitions of an almost paracontact metric structure, para $\alpha$-Sasakian structure and para $\alpha$-contact structure. We introduce the notion of an induced almost paracontact structure and prove some results related to this structure.
\par Section 4 contains main results of this paper. We prove that if $(\varphi, \xi ,\eta,h)$ is an almost paracontact metric structure then the hypersurface is equiaffine and the shape operator $S=-\id$. In consequence, the structure is para $(-1)$-Sasakian. We also prove that the hypersurface is a piece of a hyperquadric and give an explicit formula for it .

\section{Preliminaries}
We briefly recall the basic formulas for affine differential
geometry. For more details, we refer to \cite{NS}. Let $f\colon M\rightarrow\R^{n+1}$ be an orientable
connected differentiable $n$-dimensional hypersurface immersed in
affine space $\R^{n+1}$ equipped with its usual flat connection
$\D$. Then for any transversal vector field $C$ we have
$$
\D_Xf_\ast Y=f_\ast(\nabla_XY)+h(X,Y)C
$$
and
$$
\D_XC=-f_\ast(SX)+\tau(X)C,
$$
where $X,Y$ are tangent vector fields. For any transversal vector field $\nabla$ is a torsion-free connection, $h$ is a symmetric
bilinear form on $M$, called the second
fundamental form, $S$ is a tensor of type $(1,1)$, called the
shape operator and $\tau$ is a 1-form.
\par In this paper we assume that $h$ is nondegenerate so that $h$ defines a
pseudo-Rie\-man\-nian metric on $M$. If $h$ is nondegenerate, then we say that the hypersurface or the
hypersurface immersion is \emph{nondegenerate}. We have the following
\begin{tw}[\cite{NS},Fundamental equations]\label{tw::FundamentalEquations}
For an arbitrary transversal vector field $C$ the induced
connection $\nabla$, the second fundamental form $h$, the shape
operator $S$, and the 1-form $\tau$ satisfy
the following equations:
\begin{align}
\label{eq::Gauss}&R(X,Y)Z=h(Y,Z)SX-h(X,Z)SY,\\
\label{eq::Codazzih}&(\nabla_X h)(Y,Z)+\tau(X)h(Y,Z)=(\nabla_Y h)(X,Z)+\tau(Y)h(X,Z),\\
\label{eq::CodazziS}&(\nabla_X S)(Y)-\tau(X)SY=(\nabla_Y S)(X)-\tau(Y)SX,\\
\label{eq::Ricci}&h(X,SY)-h(SX,Y)=2d\tau(X,Y).
\end{align}
\end{tw}
The equations (\ref{eq::Gauss}), (\ref{eq::Codazzih}),
(\ref{eq::CodazziS}), and (\ref{eq::Ricci}) are called the
equation of Gauss, Codazzi for $h$, Codazzi for $S$ and Ricci,
respectively.
\par For an affine hypersurface the cubic form $Q$ is
defined by the formula
\begin{equation}\label{eq::CubicForm}
Q(X,Y,Z)=(\nabla_Xh)(Y,Z)+\tau(X)h(Y,Z).
\end{equation}
It follows from the equation of Codazzi (\ref{eq::Codazzih}) that
$Q$ is symmetric in all three variables.
\par For a hypersurface immersion $f\colon M\rightarrow \R^{n+1}$
a transversal vector field $C$ is said to be \emph{equiaffine}
(resp. \emph{locally equiaffine}) if $\tau=0$ (resp. $d\tau=0$).
\par Let $\dim M=2n+1$ and $f\colon M\rightarrow \R^{2n+2}$ be
a nondegenerate (relative to the second fundamental form) affine hypersurface. We always assume that $\R^{2n+2}$ is endowed with
the standard paracomplex structure $\widetilde{J}$
$$
\widetilde{J}(x_1,\ldots,x_{n+1},y_1,\ldots,y_{n+1})=(y_1,\ldots,y_{n+1},x_1,\ldots,x_{n+1}).
$$
Let $C$ be a transversal vector field on $M$. We say that $C$ is
\emph{$\widetilde{J}$-tangent} if $\widetilde{J}C_x\in f_\ast(T_xM)$ for every $x\in M$.
We also define a distribution $\DD$ on $M$ as the biggest $\widetilde{J}$-invariant distribution on $M$, that is
$$
\DD_x=f_\ast^{-1}(f_\ast(T_xM)\cap \widetilde{J}(f_\ast(T_xM)))
$$
for every $x\in M$. We have that $\dim\DD _x\geq 2n$. If for some $x$ the $\dim\DD _x=2n+1$ then $\DD _x=T_xM$ and it is not possible to find $\widetilde{J}$-tangent transversal vector field in a neighbourhood of $x$. In this paper we study $f$ with a $\widetilde{J}$-tangent transversal vector field $C$, so in particular $\dim\DD=2n$. The distribution $\DD$ is smooth as the intersection of two smooth distributions and because $\dim \DD$ is constant.  A vector field $X$ is called a \emph{$\DD$-field}
if $X_x\in\DD_x$ for every $x\in M$. We use the notation $X\in\DD$ for vectors as well as for $\DD$-fields.
%It follows from \cite[Prop. 2.5]{NS} that
%the definition of $\Dh$ is independent of choice of a transversal
%vector field.
We say that the distribution $\DD$ is nondegenerate if
$h$ is nondegenerate on $\DD$.
\par Additionally, we define a 1-dimensional distribution $\Dh$ as follows
$$
{\Dh}_x:=\{X\in T_xM\colon h(X,Y)=0\;\forall\;Y\in {\DD}_x\},
$$
where $h$ is the second fundamental form on $M$ relative to any
transversal vector field.
\par To simplify the writing, we will be omitting $f_\ast$ in front of vector fields in most cases.
%From now on we always assume that $f$ is nondegenerate.
\section{Almost paracontact structures}
A $(2n+1)$-dimensional manifold $M$ is said to have an
\emph{almost paracontact structure} if there exist on $M$ a tensor
field $\varphi$ of type (1,1), a vector field $\xi$ and a 1-form
$\eta$ which satisfy
\begin{align}
\varphi^2(X)&=X-\eta(X)\xi,\\
\eta(\xi)&=1
\end{align}
for every $X\in TM$
and the tensor field $\varphi$ induces an almost paracomplex structure on the distribution $\DD=\operatorname{ker}\eta$. That is the eigendistributions $\DD ^{+},\DD ^{-}$ corresponding to the eigenvalues $1,-1$
of $\varphi$ have equal dimension $n$.
 If additionally there is a pseudo-Riemannian metric
$g$ on $M$ of signature $(n+1,n)$ such that
\begin{align}\label{Def::eq::metric}
g(\varphi X,\varphi Y)=-g(X,Y)+\eta(X)\eta(Y)
\end{align}
for every $X,Y\in TM$ then $(\varphi,\xi,\eta,g)$ is called an
\emph{almost paracontact metric structure}. In particular, for an almost paracontact metric structure we have
\begin{align}
\label{eq::etaXhXsi}\eta(X)=g(X,\xi)
\end{align}
for all $X\in TM$. Hence $\xi$ is $g$-orthogonal to $\DD$.

An almost paracontact metric
structure is called  \emph{ para $\alpha$-Sasakian} if
\begin{align}
(\widehat{\nabla}_X\varphi)(Y)=\alpha(-g(X,Y)\xi+\eta(Y)X),
\end{align}
where $\widehat{\nabla}$ is the Levi-Civita connection for $g$ and $\alpha$ is some smooth function on $M$.
In particular, when $\alpha=1$ we get the standard para-Sasakian structure.
An almost paracontact metric manifold is called \emph{para $\alpha$-contact} if
\begin{align}\label{eq::deta}
d\eta (X,Y)=\alpha g(X,\varphi Y)
\end{align}
for a certain non-zero function $\alpha$ and
for every $X,Y\in TM$.
When $\alpha =1$ an almost paracontact metric structure $(\varphi,\xi,\eta,g)$ satysfying (\ref{eq::deta})  is called a
\emph{paracontact metric structure}.

We say that an almost paracontact structure $(\varphi,\xi,\eta)$ is
\emph{normal} if
\begin{align}\label{Def::eq::normality}
[\varphi,\varphi]-2d\eta\otimes\xi=0,
\end{align}
where $[\varphi,\varphi]$ is the Nijenhuis tensor for $\varphi$.
We have the following theorem
\begin{tw}[\cite{OL}]\label{tw::equivalentforparaalfasasakian}
An almost paracontact metric manifold is para $\alpha$-Sasakian if and only if it is normal and para $\alpha$-contact.
\end{tw}
\par Let $f\colon M\rightarrow \R^{2n+2}$ be a nondegenerate affine
hypersurface with a $\widetilde{J}$-tangent transversal vector field $C$. Then
we can define a vector field $\xi$, a 1-form $\eta$ and a tensor field
$\varphi$ of type (1,1) as follows:
\begin{align}
&\xi:=\widetilde{J}C;\\
\label{etanaD::eq::0}&\eta|_\DD=0 \text{ and } \eta(\xi)=1; \\
&\varphi|_\DD=\widetilde{J}|_\DD \text{ and } \varphi(\xi)=0.
\end{align}
It is easy to see that $(\varphi,\xi,\eta)$ is an almost paracontact
structure on $M$. This structure will be called the \emph{induced almost
paracontact structure}.
\par We shall now prove
\begin{tw}\label{tw::Wzory}
Let $f\colon M\rightarrow \mathbb{R}^{2n+2}$ be an affine hypersurface with a $\widetilde{J}$-tangent transversal vector field $C$.
If $(\varphi,\xi,\eta)$ is an induced almost paracontact structure on $M$
then the following equations hold:
\begin{align}
\label{Wzory::eq::1}&\eta(\nabla_XY)=h(X,\varphi Y)+X(\eta(Y))+\eta(Y)\tau(X),\\
\label{Wzory::eq::2}&\varphi(\nabla_XY)=\nabla_X\varphi Y-\eta(Y)SX-h(X,Y)\xi,\\
\label{Wzory::eq::3}&\eta([X,Y])=h(X,\varphi Y)-h(Y,\varphi X)+X(\eta(Y))-Y(\eta(X))\\
\nonumber &\qquad\qquad\quad+\eta(Y)\tau(X)-\eta(X)\tau(Y),\\
\label{Wzory::eq::4}&\varphi([X,Y])=\nabla_X\varphi Y-\nabla_Y\varphi X+\eta(X)SY-\eta(Y)SX,\\
\label{Wzory::eq::5}&\eta(\nabla_X\xi)=\tau(X),\\
\label{Wzory::eq::6}&\eta(SX)=-h(X,\xi)
\end{align}
for every $X,Y\in \X(M)$.
\end{tw}
\begin{proof}
For every $X\in TM$ we have
$$
\widetilde{J}X=\varphi X+\eta(X)C.
$$
We also have
\begin{align*}
\widetilde{J}(D_XY)&=\widetilde{J}(\nabla_XY+h(X,Y)C)=\widetilde{J}(\nabla_XY)+h(X,Y)\widetilde{J}C\\
&=\varphi(\nabla_XY)+\eta(\nabla_XY)C+h(X,Y)\xi
\end{align*}
and
\begin{align*}
D_X\widetilde{J}Y&=D_X(\varphi Y+\eta(Y)C)=D_X\varphi
Y+X(\eta(Y))C+\eta(Y)D_XC\\
&=\nabla_X\varphi Y+h(X,\varphi
Y)C+X(\eta(Y))C+\eta(Y)(-SX+\tau(X)C)\\
&=\nabla_X\varphi Y-\eta(Y)SX+(h(X,\varphi
Y)+X(\eta(Y))+\eta(Y)\tau(X))C.
\end{align*}
Since $\D_X\widetilde{J}Y=\widetilde{J}(D_XY)$, comparing transversal and tangent parts, we obtain (\ref{Wzory::eq::1}) and (\ref{Wzory::eq::2}), respectively. Equations
(\ref{Wzory::eq::3})---(\ref{Wzory::eq::6}) follow directly from (\ref{Wzory::eq::1}) and (\ref{Wzory::eq::2}).
\end{proof}
From the above theorem we immediately get
\begin{wn}\label{wn::Wzory} For every $Z,W\in \DD$ we have
\begin{align}
\label{WniosekWzory::eq::1}&\eta(\nabla_ZW)=h(Z,\varphi W),\\
\label{WniosekWzory::eq::2}&\eta(\nabla_\xi Z)=h(\xi,\varphi Z),\\
\label{WniosekWzory::eq::3}&\varphi(\nabla_ZW)=\nabla_Z\varphi W-h(Z,W)\xi,\\
\label{WniosekWzory::eq::4}&\eta([Z,W])=h(Z,\varphi W)-h(W,\varphi Z),\\
\label{WniosekWzory::eq::5}&\eta([Z,\xi])=-h(\xi,\varphi Z)+\tau(Z).
\end{align}
\end{wn}
Almost paracontact normal structures can be characterized as follows
\begin{stw}\label{tw::WarunekNaNormalnosc}
The induced almost paracontact structure $(\varphi,\xi,\eta)$ is normal if and only if
$$
S\varphi Z-\varphi SZ+\tau(Z)\xi=0\qquad \text{for every $Z\in
\mathcal{D}$}.
$$
\end{stw}
\begin{proof}
It is an immediate consequence of (\ref{Def::eq::normality}), the identity
$$
d\eta (X,Y)=\frac{1}{2}(X(\eta(Y))-Y(\eta(X))-\eta([X,Y]))
$$ and the formulas (\ref{Wzory::eq::3}) and (\ref{Wzory::eq::4}).
\end{proof}
\section{Main results}
In this section we always assume that $(\varphi,\xi,\eta)$ is an induced almost paracontact structure.
In order to prove the main theorem of this section we need the following two lemmas:

\begin{lm}\label{lm::wzory::EST}
If $(\varphi ,\xi ,\eta ,h)$ is an almost paracontact metric structure then
\begin{align}
\label{WniosekWzory::eq::6}&\eta (X)=h(X,\xi),\text{ for every } X\in TM,\\
\label{WniosekWzory::eq::7}&S(\DD)\subset \DD,\\
\label{WniosekWzory::eq::8}&S\xi =-\xi +Z_0, \text{ where } Z_0\in\DD \\
\label{WniosekWzory::eq::9}&\tau (Z)=-h(Z,\varphi Z_0) \text{ for every } Z\in\DD .
\end{align}
\end{lm}
\begin{proof}
Properties (\ref{WniosekWzory::eq::6}), (\ref{WniosekWzory::eq::7}) and (\ref{WniosekWzory::eq::8}) are immediate consequence of (\ref{Def::eq::metric}) and (\ref{Wzory::eq::6}). The Codazzi equation for $S$ implies that
$$
\nabla _XS\xi-S(\nabla_X\xi)-\tau (X) S\xi=\nabla_{\xi}SX-S(\nabla_{\xi}X)-\tau(\xi)SX.
$$
Formula (\ref{WniosekWzory::eq::2}) and the fact that $(\varphi ,\xi ,\eta ,h)$ is metric structure imply that $\nabla_{\xi}Z\in\DD$ for every $Z\in\DD$.
By (\ref{WniosekWzory::eq::7}) and (\ref{WniosekWzory::eq::8}) we obtain
\begin{align}\label{lm::eq::tauZ}
\tau (Z)=-\eta (\nabla _ZZ_0)+\eta (\nabla _Z\xi)+\eta (S(\nabla _Z\xi))
\end{align} for every $Z\in \DD$. Now, using (\ref{Wzory::eq::1}), (\ref{Wzory::eq::6}), (\ref{WniosekWzory::eq::6}) and the fact that $(\varphi ,\xi,\eta,h)$ is metric structure we get
$$
\eta (\nabla _ZZ_0)=h(Z,\varphi Z_0), \quad \eta (S(\nabla _Z\xi))=-\eta (\nabla_Z\xi)
$$ for every $Z\in\DD$. Hence, equation (\ref{lm::eq::tauZ}) can be rewritten as
$$
\tau (Z)=-h(Z,\varphi Z_0),
$$ which proves (\ref{WniosekWzory::eq::9}).
\end{proof}

\begin{lm}\label{lm::FormaKubiczna} If $(\varphi,\xi,\eta,h)$ is
an almost paracontact metric structure then
\begin{align}
\label{FormaKubiczna::eq::1}&Q(X,W,Z)=-Q(X,\varphi W,\varphi Z),\\
\label{FormaKubiczna::eq::2}&Q(W_1,W_2,W_3)=0,\\
\label{FormaKubiczna::eq::3}&Q(\xi,W,W)=-h(SW,\varphi W)=h(S\varphi W,W)
\end{align}
for every $X\in \X(M)$ and  $W,W_1,W_2,W_3,Z\in\DD$.
\end{lm}
\begin{proof}
Let $X\in \X(M)$ and $W,Z\in\DD$. Then by (\ref{eq::CubicForm}) and (\ref{Def::eq::metric}) we have
\begin{align*}
Q(X,\varphi W,\varphi Z)&=X(h(\varphi W,\varphi
Z))-h(\nabla_X\varphi W,\varphi Z)-h(\varphi W,\nabla_X\varphi
Z)\\&\hspace{12pt}+\tau(X)h(\varphi
W,\varphi Z)\\
&=-X(h(W,Z))-h(\nabla_X\varphi W,\varphi Z)-h(\varphi
W,\nabla_X\varphi Z)\\&\hspace{12pt}-\tau(X)h(W,Z).
\end{align*}
By Theorem \ref{tw::Wzory} we see that
$$
\nabla_X\varphi W=\varphi(\nabla_XW)+h(X,W)\xi
$$
and
$$
\nabla_X\varphi Z=\varphi(\nabla_XZ)+h(X,Z)\xi.
$$
 Thus using the above and (\ref{Def::eq::metric}) we get
\begin{align*}
Q(X,\varphi W,\varphi Z)&=-X(h(W,Z))-h(\varphi (\nabla _XW),\varphi Z)-h(X,W)h(\xi,\varphi Z)\\
&\hspace{12pt}-h(\varphi W,\varphi (\nabla _XZ))-h(X,Z)h(\varphi W,\xi)-\tau (X)h(W,Z)\\
&=-X(h(W,Z))+h(\nabla_XW,Z)+h(W,\nabla_XZ)-\tau(X)h(W,Z)\\
&=-Q(X,W,Z),
\end{align*}
which proves (\ref{FormaKubiczna::eq::1}). To prove (\ref{FormaKubiczna::eq::2}) observe
that from (\ref{FormaKubiczna::eq::1}) we have
\begin{align*}
Q(W,W,W)=-Q(W,\varphi W,\varphi W)=-Q(\varphi W,\varphi W,W)=0
\end{align*}
for every  $W\in\DD$, because $h(\varphi W,W)=0$ and $\nabla _{\varphi W}W=\varphi (\nabla _{\varphi W}\varphi W)-h(W,W)\xi$. Since $Q$ is symmetric in all three variables,
the last equation implies that  $Q(W_1,W_2,W_3)=0$ for every $W_1,W_2,W_3\in\DD$.
To prove (\ref{FormaKubiczna::eq::3}) note first that
\begin{align*}
Q(\xi,W,W)=Q(W,\xi,W)=-h(\nabla_W\xi,W)-h(\xi,\nabla_WW),
\end{align*}
since $h(\xi ,W)=0$. Formula (\ref{Wzory::eq::2}) imply that
$$
\varphi(\nabla_W\xi)=-SW.
$$
From (\ref{Def::eq::metric}) and (\ref{WniosekWzory::eq::1}) we get
$$\nabla_WW\in\DD.$$
Now we have
$$Q(\xi,W,W)=-h(\nabla _W\xi,W)=h(\varphi (SW),W)=-h(SW,\varphi W)$$
for every $W\in\DD$.
From (\ref{FormaKubiczna::eq::1}) we obtain
$$Q(\xi,W,W)=-Q(\xi,\varphi W,\varphi W)$$
and in consequence
$$-h(SW,\varphi W)=h(S\varphi W,W),$$
which completes the proof of (\ref{FormaKubiczna::eq::3}).
\end{proof}

We shall now prove
\begin{tw} \label{tw::SiTau=0}
Let $f\colon M\rightarrow \R^{2n+2}$ be a nondegenerate hypersurface with a $\widetilde{J}$-tangent transversal vector field
and let $(\varphi,\xi,\eta)$ be the induced almost paracontact structure on $M$.
If $(\varphi,\xi,\eta,h)$ is the almost paracontact metric structure
then
$$
S=-\id\qquad\text{and}\qquad \tau=0.
$$
\end{tw}
\begin{proof}
Let $W,Z\in\DD$. Formulas (\ref{eq::CodazziS}), (\ref{WniosekWzory::eq::7}) and (\ref{etanaD::eq::0}) imply
that
$$
\eta(\nabla_WSZ)-\eta(S(\nabla_WZ))=\eta(\nabla_ZSW)-\eta(S(\nabla_ZW)).
$$
Thus, by (\ref{Wzory::eq::6}) and (\ref{eq::etaXhXsi}),
$$
\eta(\nabla_WSZ)-\eta(\nabla_ZSW)=\eta(S([W,Z]))=-\eta([W,Z])=\eta ([Z,W]).
$$
By Corollary  \ref{wn::Wzory} (the formulas (\ref{WniosekWzory::eq::1}), (\ref{WniosekWzory::eq::4})) we get
$$
h(W,\varphi SZ)-h(Z,\varphi SW)=-h(W,\varphi Z)+h(Z,\varphi W).
$$
Replacing $Z$ with $\varphi Z$ and using the fact that
$(\varphi,\xi,\eta,h)$ is a metric structure we have
\begin{equation}\label{eq::PierwszyWzor}
-h(\varphi W,S\varphi Z)+h(Z,SW)=-2h(W,Z)\quad\text{for every $W,Z\in\DD$}.
\end{equation}
Using the Gauss equation we get
\begin{align}\label{eq::SiTauA}
(R(W,\varphi W)\cdot h)(\varphi W,\varphi W)&=-2h(R(W,\varphi
W)\varphi W,\varphi W)\\\nonumber &=2h(W,W)h(SW,\varphi W)
\end{align}
for every $W\in\DD$. On the other hand
\begin{align*}
(R(W,\varphi W)\cdot h)(\varphi W,\varphi
W)=(\nabla_W\nabla_{\varphi W}h)(\varphi W,\varphi
W)\\-(\nabla_{\varphi W}\nabla_{W}h)(\varphi W,\varphi
W)-(\nabla_{[W,\varphi W]}h)(\varphi W,\varphi W).
\end{align*}
The following formulas are obvious:
\begin{align*} % SZTUCZKA Z \hspace{-2pt} !!!!
(\nabla_W\hspace{-2pt}\nabla_{\varphi W}h)(\varphi W,\varphi
W)\hspace{-1pt}=\hspace{-1pt}W((\nabla_{\varphi W}h)(\varphi W,\varphi W))-\hspace{-2pt}2(\nabla_{\varphi
W}h)(\nabla_W\varphi W,\varphi W),\\
(\nabla_{\varphi W}\hspace{-2pt}\nabla_{W}h)(\varphi W,\varphi W)\hspace{-1pt}=\hspace{-1pt}{\varphi
W}((\nabla_{W}h)(\varphi W,\hspace{-1pt}\varphi W))-\hspace{-2pt}2(\nabla_{W}h)(\nabla_{\varphi
W}\varphi W,\hspace{-1pt}\varphi W).
\end{align*}
We have
$$
(\nabla_Xh)(Y,Z)=Q(X,Y,Z)-\tau(X)h(Y,Z)
$$
for every $X,Y,Z\in \X(M)$.
Thus Lemma \ref{lm::FormaKubiczna} and the above formulas imply
\begin{align*}
(\nabla_W\nabla_{\varphi W}h)&(\varphi W,\varphi W)=W\Big(Q(\varphi
W,\varphi W,\varphi W)-\tau(\varphi W)h(\varphi W,\varphi
W)\Big )\\&\hspace{12pt}-2Q(\varphi W,\nabla_W\varphi W,\varphi W)+2\tau(\varphi
W)h(\nabla_W\varphi W,\varphi W)\\
&=-W(\tau(\varphi W))h(\varphi W,\varphi W)-\tau(\varphi W)W(h(\varphi W,\varphi
W))\\&\hspace{12pt}+2Q(\nabla_W\varphi W,W,W)+2\tau(\varphi W)h(\nabla_W\varphi
W,\varphi W)\\
&=W(\tau(\varphi W))h(W,W)-\tau(\varphi W)(\nabla_Wh)(\varphi
W,\varphi W)\\
&\hspace{12pt}+2h(W,W)Q(\xi,W,W)\\
&=W(\tau(\varphi W))h(W,W)-\tau(\varphi
W)\tau(W)h(W,W)\\&\hspace{12pt}+2h(W,W)Q(\xi,W,W)
\end{align*}
where, in the last equality, we used (\ref{WniosekWzory::eq::1}). In a similar way we obtain
\begin{align*}
(\nabla_{\varphi W}\nabla_{W}h)&(\varphi W,\varphi W)=\varphi
W(Q(W,\varphi W,\varphi W)-\tau(W)h(\varphi W,\varphi
W))\\&\hspace{12pt}-2Q(W,\nabla_{\varphi W}\varphi W,\varphi
W)+2\tau(W)h(\nabla_{\varphi W}\varphi W,\varphi W)\\
&=-\varphi W(\tau(W))h(\varphi W,\varphi W)-\tau(W)\varphi W(h(\varphi W,\varphi W
))\\&\hspace{12pt}+2\tau(W)h(\nabla_{\varphi W}\varphi W,\varphi W)\\
&=\varphi W(\tau(W))h(W,W)-\tau(W)(\nabla_{\varphi W}h)(\varphi
W,\varphi W)\\
&=\varphi W(\tau(W))h(W,W)-\tau(W)\tau(\varphi W)h(W,W).
\end{align*}
From (\ref{WniosekWzory::eq::4}) and Lemma \ref{lm::FormaKubiczna} we also have
\begin{align*}
(\nabla_{[W,\varphi W]}h)(\varphi W,\hspace{-1pt}\varphi W)\hspace{-1pt}&=\hspace{-1pt}Q([W,\hspace{-1pt}\varphi
W],\hspace{-1pt}\varphi W,\hspace{-1pt}\varphi W)\hspace{-1pt}-\hspace{-1pt}\tau([W,\varphi W])h(\varphi W,\hspace{-1pt}\varphi
W)\\
&=-\eta([W,\varphi W])Q(\xi,W,W)+\tau([W,\varphi W])h(W,W)\\
&=-2h(W,W)Q(\xi,W,W)+\tau([W,\varphi W])h(W,W).
\end{align*}
Using (\ref{FormaKubiczna::eq::3}) and the Ricci equation (\ref{eq::Ricci}), we get
\begin{align}\label{eq::SiTauB}
2Q(\xi,W,W)=-h(SW,\varphi W)+h(W,S\varphi W)=2d\tau(W,\varphi W).
%(R(W,\varphi W)\cdot h)(\varphi W,\varphi W)=-2 d \tau(W,\varphi
%W)h(W,W)\\\nonumber-4Q(\xi,W,W)h(W,W)=(h(SW,\varphi
%W)-h(W,S\varphi W))h(W,W)\\\nonumber+4h(SW,\varphi W)h(W,W)
%=6h(SW,\varphi W)h(W,W)
\end{align}
From (\ref{eq::SiTauB}) and the preceding formulas, we obtain
$$
(R(W,\varphi W)\cdot h)(\varphi W,\varphi W)=6d\tau(W,\varphi W)h(W,W)
$$
and so, by (\ref{eq::SiTauB}) and (\ref{FormaKubiczna::eq::3}),
$$
(R(W,\varphi W)\cdot h)(\varphi W,\varphi W)=6Q(\xi,W,W)h(W,W)=-6h(W,W)h(SW,\varphi W),
$$
which, combined with (\ref{eq::SiTauA}), yields
$$
h(W,W)\cdot h(SW,\varphi W)=0 %h(W,W)=0
$$
for every $W\in\DD$. Using the fact that $h$ is nondegenerate on $\DD$ we get
\begin{equation}\label{eq::SiTauC}
h(SW,\varphi W)=0 %h(W,W)=0
\end{equation} for every $W\in \DD$.

Now (\ref{eq::SiTauC}) implies
$$
0=h(S(W+2\varphi Z),\varphi W+2 Z)=2h(SW,Z)+2h(S\varphi
Z,\varphi W).
$$
Therefore
$$
h(S\varphi Z,\varphi W)=-h(SW,Z).
$$
By (\ref{eq::PierwszyWzor}) we also have
$$
h(S\varphi Z,\varphi W)=2h(W,Z)+h(SW,Z).
$$
The above formulas imply that

$$
h(SW,Z)=-h(W,Z)
$$
for every $Z\in\DD$. Now, since $\DD$ is nondegenerate and by (\ref{WniosekWzory::eq::7}) it follows that
$$
SW=-W
$$
for every $W\in\DD$. From Lemma \ref{lm::wzory::EST} we get
\begin{align}\label{eq::SiTau::PostacS}
SX=-X+\eta (X)Z_0
\end{align} for every $X\in X(M)$.
We shall show that $Z_0=0$. Assume $Z_0\neq 0$, then using the Codazzi equation for $S$ we have
$$
\nabla_WSZ_0-S(\nabla_WZ_0)-\tau(W)SZ_0=\nabla_{Z_0}SW-S(\nabla_{Z_0}W)-\tau(Z_0)SW.
$$
Since $\tau(Z_0)=0$ (Lemma \ref{lm::wzory::EST}), using (\ref{eq::SiTau::PostacS}) we can
rewrite the above equality in the form
$$
-\eta(\nabla_WZ_0)Z_0+\tau(W)Z_0=-\eta(\nabla_{Z_0}W)Z_0,
$$
Now, by (\ref{WniosekWzory::eq::4}) and (\ref{WniosekWzory::eq::9}) we have
$$
-\tau(W)Z_0=\eta([Z_0,W])Z_0=-2h(W,\varphi Z_0)Z_0=2\tau(W)Z_0.
$$
The last equality implies that $\tau|_\DD=0$. Now, formula (\ref{WniosekWzory::eq::9}) implies
$Z_0=0$ which contradicts our assumption.
\par The property $\tau=0$ easily follows from the fact that $S=-\id$ and the equation of Codazzi for $S$.
The proof is completed.
\end{proof}
The following theorem gives equivalent conditions for being the induced almost paracontact metric structure.
\begin{tw}
Let $f\colon M\rightarrow \R^{2n+2}$ be a nondegenerate hypersurface with a $\widetilde{J}$-tangent transversal vector field
and let $(\varphi,\xi,\eta)$ be the induced almost paracontact structure on $M$.
The following conditions are equivalent:
\begin{align}
&\text{$(\varphi,\xi,\eta,h)$ is an almost paracontact metric structure},\\
&\text{$(\varphi,\xi,\eta,h)$ is a para $\alpha$-contact metric structure, where $\alpha =-1$},\\
&\text{$(\varphi,\xi,\eta,h)$ is a para $\alpha$-Sasakian structure, where $\alpha =-1$}.
\end{align}
\end{tw}
\begin{proof} If $(\varphi,\xi,\eta,h)$ is an almost paracontact metric structure then
by Theorem \ref{tw::SiTau=0} we obtain $\tau=0$. Theorem \ref{tw::Wzory} (eq. (\ref{Wzory::eq::3})) implies that
$(\varphi,\xi,\eta,h)$ is a para $(-1)$-contact metric structure. Again by Theorem \ref{tw::SiTau=0} we get $S=-\id$. Hence
$(\varphi,\xi,\eta)$ is normal (Prop. \ref{tw::WarunekNaNormalnosc}). Now Theorem \ref{tw::equivalentforparaalfasasakian} completes
the proof.
\end{proof}
Using Pick-Berwald theorem we get
\begin{tw} Let $f\colon M\rightarrow \R^{2n+2}$ be a nondegenerate hypersurface with a $\widetilde{J}$-tangent transversal vector field
and let $(\varphi,\xi,\eta)$ be the induced almost paracontact structure on $M$.
If $(\varphi,\xi,\eta,h)$ is the almost paracontact metric
structure, then $f(M)$ is a piece of a hyperquadric.
\end{tw}
\begin{proof}
It is enough to show that $Q\equiv 0$. By Lemma \ref{lm::FormaKubiczna} we have
$$
Q(W_1,W_2,W_3)=0 \quad\text{for every $W_1,W_2,W_3\in\DD$}
$$
and
$$
Q(\xi,W_1,W_2)=0\quad\text{for every $W_1,W_2\in\DD$}.
$$
Since $\tau=0$ by Theorem \ref{tw::SiTau=0}, using (\ref{Wzory::eq::1}) and (\ref{WniosekWzory::eq::6}) we obtain
$$
Q(X,\xi,\xi)=-2h(\nabla_X\xi,\xi)=-2\eta(\nabla_X\xi)=0
$$
for every $X\in\X(M)$. The above equalities imply that
$$
Q(X_1,X_2,X_3)=0
$$
for every $X_1,X_2,X_3\in\X(M)$.
\end{proof}
Finally, we can find an explicit formula for such hyperquadrics. We have the following theorem
\begin{tw}
The nondegenerate hyperquadric of center $0$ such that the induced almost paracontact structure $(\varphi, \xi, \eta)$ is metric relative to the second fundamental form $h$ can be expressed in the form
\begin{align*}
H=\{x\in\R^{2n+2}\colon x^{T}Ax=1\},
\end{align*}
where $\det A\neq 0$ and
$$
A=\left[
\begin{matrix}
P & R \\
-R & -P
\end{matrix}
\right],
$$
$P^T=P$, $R^T=-R$, $P,R\in M(n+1,n+1,\mathbb{R})$.
\par Moreover, the induced almost paracontact structure for hyperquadrics of the above form is metric relative to the second fundamental form.
\end{tw}
\begin{proof}
Every nondegenerate hyperquadric of center $0$ has a form
\begin{align*}
H=\{x\in\R^{2n+2}\colon x^{T}Ax=1\},
\end{align*}
where $\det A\neq 0$, $A^T=A$, $A\in M(2n+2,2n+2,\mathbb{R})$.
Since $(\varphi, \xi, \eta)$ is the induced metric structure, Theorem \ref{tw::SiTau=0} implies that the shape operator $S=-\id$ and $\tau=0$.
So $C:=x$ is a $\widetilde{J}$-tangent transversal vector field and $\xi=\widetilde{J}x$ is tangent. Since $Ax$ is orthogonal (relative to the standard inner product $\langle\cdot,\cdot\rangle$ on $\R^{2n+2}$) to $H$ we have
$$
0=\langle \widetilde{J}x, Ax\rangle = x^T\widetilde{J}Ax
$$
for every $x\in H$.
We also have
$$
x^TA\widetilde{J}x=0,
$$
so in consequence
\begin{align}
\label{hgjajt::eq::4}&x^{T}(\widetilde{J}A+A\widetilde{J})x=0
\end{align}
for every $x\in H$.
Since $\widetilde{J}A+A\widetilde{J}$ is symmetric, formula (\ref{hgjajt::eq::4}) implies that
\begin{align*}
\widetilde{J}A=-A\widetilde{J}.
\end{align*}

The last formula implies that
$$
A=\left[
\begin{matrix}
P & R \\
-R & -P
\end{matrix}
\right],
$$
$P^T=P$, $R^T=-R$, $P,R\in M(n+1,n+1,\mathbb{R})$.
\par Now, we shall show that the induced almost paracontact structure is metric. Since $\langle Ax,C\rangle=2$, it is enough to prove that
\begin{align*}
\langle D_{\widetilde{J}Z}\widetilde{J}W,Ax \rangle=-\langle D_{Z}W,Ax \rangle
\end{align*} for every $Z,W\in\DD$.
For $Z,W\in\DD$ we have
\begin{align}
 \label{eq::8}&\langle Z ,Ax \rangle =0,\\
  \label{eq::9}&\langle \widetilde{J}Z,Ax\rangle =0,\\
   \label{eq::10}&\langle W,Ax\rangle =0,\\
    \label{eq::11}&\langle \widetilde{J}W,Ax\rangle =0.
\end{align}
We also have
 \begin{align}
 \label{j::eq::12}&\langle \widetilde{J}X,Y\rangle =\langle X,\widetilde{J}Y\rangle
 \end{align} for every $X,Y$ tangent to hyperquadric. Using the fact that $DA=0$ we obtain
 \begin{align}
 \label{D::eq::13}&D_ZAx=AZ
 \end{align} and
 \begin{align}
 \label{DJ::eq::14}&D_{\widetilde{J}Z}Ax=A\widetilde{J}Z.
 \end{align}
Since  $D\langle\cdot,\cdot\rangle = 0$ we also have
 \begin{equation*}
 \langle D_{\widetilde{J}Z}\widetilde{J}W,Ax\rangle =\widetilde{J}W(\langle \widetilde{J}Z,Ax\rangle )-\langle \widetilde{J}W,D_{\widetilde{J}Z}Ax\rangle .
 \end{equation*}
 Using (\ref{eq::9}), (\ref{DJ::eq::14}) and (\ref{j::eq::12}) we obtain
 \begin{align*}
 \langle D_{\widetilde{J}Z}\widetilde{J}W,Ax\rangle  =-\langle \widetilde{J}W,A\widetilde{J}Z\rangle =\langle \widetilde{J}W,\widetilde{J}AZ\rangle
 =\langle W,AZ\rangle
 \end{align*}
On the other hand (\ref{D::eq::13}), (\ref{eq::10}) and $D\langle\cdot,\cdot\rangle = 0$ imply
 \begin{equation*}
 \langle W,AZ\rangle =\langle W,D_ZAx\rangle=-\langle D_ZW,Ax\rangle,
 \end{equation*}
 what completes the proof.
\end{proof}
\emph{This Research was financed by the Ministry of Science and Higher Education of the Republic of Poland.}

\bibliographystyle{aplain}

\vspace{1cm}
\par \ \\
Zuzanna Szancer\\
Department of Applied Mathematics, \\
University of Agriculture in Krakow, \\
253 Balicka St., 30-198 Krakow, Poland\\
e-mail: Zuzanna.Szancer@urk.edu.pl

\end{document}